\documentclass[12pt,a4paper,oneside]{amsart}

\usepackage{amsfonts, amsmath, amssymb, amsthm, amscd}
\usepackage{hyperref}
\hypersetup{
  colorlinks   = true, 
  urlcolor     = blue, 
  linkcolor    = blue, 
  citecolor   = red 
}
\usepackage{anysize}

\newtheorem{theorem}{Theorem}[section]
\newtheorem{lemma}[theorem]{Lemma}

\newtheorem{proposition}[theorem]{Proposition}
\newtheorem{conjecture}[theorem]{Conjecture}

\theoremstyle{definition}
\newtheorem{definition}[theorem]{Definition}

\theoremstyle{remark}
\newtheorem{remark}[theorem]{Remark}
\newtheorem{question}[theorem]{Question}

\numberwithin{equation}{section}

\DeclareMathOperator{\dist}{dist}
\DeclareMathOperator{\conv}{conv}
\DeclareMathOperator{\lk}{lk}

\newcommand{\const}{{\rm const}}

\newcommand{\inte}{\mathop{\rm int}}
\newcommand{\cl}{\mathop{\rm cl}}
\renewcommand{\epsilon}{\varepsilon}
\renewcommand{\phi}{\varphi}
\renewcommand{\kappa}{\varkappa}

\newcommand{\gdiv}{\,\frac{\, *\, }{}\,}

\begin{document}

\title{Colorful theorems for strong convexity}

\author{Andreas~F.~Holmsen{$^\spadesuit$}}

\address{\linebreak Andreas~F.~Holmsen \hfill \hfill \linebreak  
Department of Mathematical Sciences,
KAIST,\hfill \hfill \linebreak 291 Daehak-ro, Daejeon 305-701, South Korea} 

\email{andreash@kaist.edu}

\thanks{{$^\spadesuit$}Supported by Swiss National Science Foundation Grants 200020-144531 and 200021-137574}

\author{Roman~Karasev{$^\clubsuit$}}

\thanks{{$^\clubsuit$}Supported by the Russian Foundation for Basic Research grants 15-31-20403 (mol\_a\_ved) and 15-01-99563 (A)}
\thanks{{$^\clubsuit$}Supported by ERC Advanced Research Grant No.~267195 (DISCONV)}

\address{Roman Karasev. \hfill \hfill \linebreak
Moscow Institute of Physics and Technology, Institutskiy per. 9, Dolgoprudny, Russia 141700; and 
\hfill \hfill \linebreak 
Institute for Information Transmission Problems RAS, Bolshoy Karetny per. 19, Moscow, Russia 127994}

\email{r\_n\_karasev@mail.ru}
\urladdr{http://www.rkarasev.ru/en/}


\begin{abstract}
We prove two colorful Carath\'eodory theorems for strongly convex hulls, generalizing the colorful Carat\'eodory theorem for ordinary convexity by Imre B\'ar\'any, the non-colorful Carath\'eodory theorem for strongly convex hulls by the second author, and the ``very colorful theorems'' by the first author and others. We also investigate if the assumption of a ``generating convex set'' is really needed in such results and try to give a topological criterion for one convex body to be a Minkowski summand of another.
\end{abstract}

\maketitle

\section{Introduction}
The colorful Carath{\'e}odory theorem, discovered by Imre B{\'a}r{\'a}ny \cite{ba1982} (and independently in a dual form by Lov{\'a}sz), states that if $X_0$, $X_1$, $\dots$, $X_n$ are subsets in $\mathbb{R}^n$, each containing the origin in their convex hulls, then there exists a system of representatives $x_0\in X_0$, $x_1\in X_1$, $\dots$, $x_n\in X_n$ such that the origin is contained in the convex hull of $\{x_0, x_1, \dots, x_n\}$. 

The classical Carath{\'e}odory theorem~\cite{car1911} is recovered by setting $X_0 = X_1 = \cdots = X_n$. There are several remarkable applications of the colorful Carath{\'e}odory theorem in discrete geometry, such as in Sarkaria's proof of Tverberg's theorem and in the proof of the existence of weak $\epsilon$-nets for convex sets. For further applications and references we recommend the reader to see Chapters 8, 9, and 10 in \cite{matousek}. We recommend the survey~\cite{eckhoff1993} for general background on Carath\'eodory's theorem and its relatives. 

Recently there have been numerous generalizations of classical results from combinatorial convexity which focus on replacing convex sets by more general subsets of $\mathbb{R}^n$ which are subject to certain topological constraints (see for instance \cite{xavi-adv, xavi-arx, km2005, KMleray, montejano} and references therein). These generalizations usually require tools from algebraic topology.  In this paper we also use such topological tools, but we do so in order to solve problems which are purely affine. 

Our goal in this paper is to give an extension of the colorful Carath\'eodory theorem~\cite{ba1982} to the notion of strong convexity and strongly convex hulls. This also generalizes the (non-colorful) Carath\'eodory theorem for strong convexity from~\cite{kar2001car}. In fact, the proof presented here is very similar to that in~\cite{kar2001car}, using the colorful topological Helly theorem~\cite{km2005} instead of the original topological Helly theorem (see for instance \cite{debrunner}); but we reproduce some of its parts because~\cite{kar2001car} appears exclusively in Russian.  

In order to state our results, we need some definitions on strong convexity from~\cite{pol1996,balashov2000}. Similar notions have appeared earlier, and we recommend the reader to~\cite{pol1996,balashov2000} for further references.

\begin{definition}
\label{definition:genset}
A convex body $K$ is a \emph{generating set} if any nonempty intersection of its translates
$$
K\gdiv T = \bigcap_{t\in T} (K - t)
$$
is a Minkowski summand of $K$, that is, $(K\gdiv T) + T' = K$ for some convex compactum $T'$.
\end{definition}

In~\cite{kar2001gen} it was shown that it is sufficient to test this property for those $T$ that contain two elements; but we do not need this simplification here. It is relatively easy to check that all two-dimensional convex bodies, Euclidean balls, and simplices in every dimension are generating sets (see for instance section 3.2 of \cite{schneider}). This property is also inherited under the Cartesian product operation. In~\cite{ivanov2007} a criterion for checking this property for $C^2$-smooth bodies was given, proving, in particular, that by sufficiently smooth perturbations of a ball one can obtain centrally symmetric generating sets which are not ellipsoids.

\begin{definition}
\label{definition:strongconv}
Let $K\subset \mathbb R^n$ be a convex body. A set $C\subset \mathbb R^n$ is called {\em $K$-strongly convex} if it is an intersection of translates of $K$, that is 
$$
C = K\gdiv T
$$
in the above notation. The minimal $K$-strongly convex set containing a given set $X$ is called its \emph{strongly convex hull}, and can be found by the following formula
$$
\conv_K X = K\gdiv (K\gdiv X).
$$
\end{definition}

Note that the $K$-strongly convex hull is only defined for those $X$ that are contained in a translate of $K$.

In~\cite{kar2001car} it was shown, assuming that $K$ is a generating set, that the strongly convex hull of $X$ equals the union of the strongly convex hulls of all subsets $X'\subset X$ such that $|X'|\le n+1$; some particular cases of this were already proved in~\cite{balashov2000}. This is the Carath\'eodory theorem for strong convexity. Here we establish the colorful Carath\'eodory theorem for strong convexity. In fact, it is possible to state it without appealing to Definition~\ref{definition:strongconv}. Let us say that a translate $K-t$ \emph{separates} a set $S$ from the origin if $K-t$ contains the set $S$ and is disjoint from the origin.

\begin{theorem} 
\label{theorem:colorfulcovering}
Let $X_0, X_1,\ldots, X_n\subset \mathbb R^n$ be finite sets and $K$ a generating set in $\mathbb R^n$. Suppose every system of representatives $x_0\in X_0$, $x_1\in X_1$, $\ldots$, $x_n\in X_n$ can be separated from the origin by a translate $K-t$. Then there exists a translate of $K$ that separates $X_i$ from the origin, for some $0\leq i  \leq n$.
\end{theorem}

This can be related to Definition~\ref{definition:strongconv} in the following way: The hypothesis means that the origin is not contained in the $K$-strongly convex hull of any system of representatives $\{x_0, x_1,\ldots, x_n\}$, and the conclusion means that the origin is not contained in the $K$-strongly convex hull of one of the $X_i$'s. This precisely generalizes the colorful Carath\'eodory theorem~\cite{ba1982} to strong convexity, of course, assuming that $K$ is a generating set.

\medskip

Throughout the paper we will be working with topological spaces that are point sets in Euclidean spaces and use their \v{C}ech cohomology. 

We denote the \v{C}ech cohomology of $X$ with coefficients in the constant sheaf $\mathbb Z$ simply by $H^*(X)$. We choose \v{C}ech cohomology because of its \emph{continuity property}, which we will use in the following form: If $Y\subseteq X$ is closed then $H^*(Y)$ is the inverse limit of $H^*(U)$ over all open neighborhoods $U\supseteq Y$. 

Note also that in our (paracompact) case, \v{C}ech cohomology coincides with the sheaf cohomology (see Theorem 5.10.1 in \cite{godement1958}), this observation is frequently used in this paper. Let us make an explicit definition:

\begin{definition}
We say that $X$ is \emph{acyclic} if $H^0(X) = \mathbb Z$ and $H^k(X)$ is zero for $k>0$.
\end{definition}

Note that every contractible subset of $\mathbb R^n$ is acyclic, but the inverse is not true. For example, the closed topologist's sine curve, which shows that connectivity does not imply arcwise connectivity, is acyclic and demonstrates that acyclicity does not imply arcwise connectivity, and therefore does not imply contractibility.

\medskip

{\em Outline of the paper.} The proof of Theorem \ref{theorem:colorfulcovering} is given in Section~\ref{section:proof1} and uses the topological colorful Helly theorem due to Kalai and Meshulam \cite{km2005}. However, in order to apply their result we need a crucial lemma concerning the topology of sets of the form $K\setminus (K+T)$ where $K$ and $T$ are arbitrary convex compacta in $\mathbb{R}^n$. This is given in Lemma \ref{lemma:acyclic}, below. In Section~\ref{section:more} we prove a ``very colorful'' version of Theorem \ref{theorem:colorfulcovering} (see Theorem~\ref{theorem:strong-caratheodory} for the precise statement). This generalization also uses Lemma \ref{lemma:acyclic}, but needs some further topological machinery, which is given in Section~\ref{section:leray-link}. In Section~\ref{section:no-cara} we give a construction of a convex body $K$ in $\mathbb{R}^3$ for which the $K$-strongly convex hull has an unbounded Carath{\'e}dory number. This shows that it is necessary to make some assumption (such as the ``generating set'' property) on the convex body $K$ in our results. In Section \ref{section:topo} we establish a partial converse to the crucial Lemma \ref{lemma:acyclic} by giving a topological criterion for a convex body $K$ to be a Minkowski summand of an open bounded convex set $T$.

\subsection*{Acknowledgments.}
The authors thank Alexey Volovikov for his explanations about the topological facts used in Section~\ref{section:topo}, Imre B\'ar\'any for the discussion which resulted in the example in Section~\ref{section:no-cara}, and the unknown referee for numerous useful remarks.

\section{Proof of Theorem~\ref{theorem:colorfulcovering}}
\label{section:proof1}






We are going to utilize the colorful topological Helly theorem from~\cite{km2005} in the following form: Assume ${\mathcal F}_0, {\mathcal F}_1, \ldots, {\mathcal F}_n$ are finite families of subsets of $\mathbb R^n$ such that the intersection of the members of any subfamily $\mathcal G\subset \bigcup_i \mathcal F_i$ is either empty or acyclic, and every intersection $\bigcap_{i=0}^n F_i$ of a system of representatives $F_0\in \mathcal F_0, F_1\in \mathcal F_1, \ldots, F_n\in \mathcal F_n$ is nonempty. Then for some $i$ the intersection $\bigcap_{F \in \mathcal F_i}F$ is nonempty.

The above result is valid in the case when all the sets $F\in \bigcup_i \mathcal F_i$ are open. This case can be deduced from the combinatorial formulation in~\cite{km2005} by comparing the cohomology of the union of the family with the cohomology of its nerve. But we are going to apply this result to sets $F$ of the form $X\setminus Y$, where $X$ and $Y$ are convex bodies, $Y$ is fixed and $X$ depends on $F$, call it $X_F$. Such sets are neither open nor closed in general, so 
some justification is needed.
First, we go to the manifold $\mathbb R^n\setminus Y$ and consider the sets $F=X_F\setminus Y$ as its closed subsets. The corollary after~\cite[Theorem 5.2.4]{godement1958} asserts that the \v{C}ech cohomology of the union of the family $\mathcal F$ can be calculated using the nerve of the family; so we can reduce the problem to studying the nerve combinatorially.

The next ingredient is the following lemma from~\cite{kar2001car}:


\begin{lemma}
\label{lemma:acyclic}
For two convex compacta $K, T\subset\mathbb R^n$ the set $K\setminus (K+T)$ is either empty or acyclic.
\end{lemma}


The proof of Lemma \ref{lemma:acyclic} was only published in Russian~\cite{kar2001car}, so we provide a sketch of the proof here for the reader's convenience.

\begin{proof}[Sketch of the proof] We start by making a series of reductions. First of all we may assume that $K$ has nonempty interior, or else we simply proceed within the affine hull of $K$. Next, we may assume that the interiors of $K$ and $K+T$ intersect, as the other case is evident. 

Now we choose a point
$$
p\in\inte K \cap \inte (K+T).
$$
For a ray $\rho$ starting at $p$, the intersection points $\rho\cap \partial K$ and $\rho\cap\partial (K+T)$ depend continuously on $\rho$. This allows us to build a homotopy equivalence between $K\setminus(K+T)$ and $(\inte K)\setminus (K+T)$; this can be done in the intersection with every $\rho$ in a way which depends continuously on $\rho$.

In the same way, working in the intersection with every $\rho$, we can, for every neighborhood $U\supseteq K\setminus (K+T)$, construct a neighborhood $U'\supseteq K\setminus (K+T)$ contained in $U$ and homeomorphic to $(\inte K)\setminus (K+T)$. This shows that the singular cohomology of $(\inte K)\setminus (K+T)$ is the same as the \v{C}ech cohomology of $K\setminus (K+T)$, and the goal is therefore to establish acyclicity of the open set $(\inte K)\setminus (K+T)$. 

For this open set, we may switch to singular cohomology and homology. It is now sufficient to represent this set as a union of an inclusion increasing sequence of acyclic open sets, since singular chains always have compact support and must fit in one of the sets of such a sequence. Represent $T$ as the intersection of an inclusion decreasing sequence of smooth and strictly convex bodies $T_k$. Then $(\inte K)\setminus (K+T)$ will be the union of the inclusion increasing sequence of sets $(\inte K)\setminus (K+T_k)$. This shows that it suffices to consider smooth and strictly convex bodies in place of $T$. 

Now we represent $K$ as the union of an inclusion increasing sequence of smooth strictly convex $K_k$, so that 
$$
K \subseteq K_k + B_{\varepsilon_k},
$$
where $B_{\varepsilon_k}$ denotes a ball of radius $\varepsilon_k>0$, for some sequence $\varepsilon_k\to 0$. We also consider $T_k = T+B_{\varepsilon_k}$. For such a choice, $(\inte K_k)\setminus(K_k+T_k)\subseteq (\inte K)\setminus(K+T)$, but the sequence $(\inte K_k)\setminus(K_k+T_k)$ is not necessarily inclusion increasing. Still, it is possible to show (we omit the details) that every compactum $X\subset (\inte K)\setminus(K+T)$ will be contained in some $(\inte K_k)\setminus(K_k+T_k)$ for sufficiently large $k$, thus reducing the question of acyclicity to smooth and strictly convex $K$ and $T$.

Now we consider the support functions
$$
s_K(p) = \sup_{x\in K}(p, x), \quad s_T(p) = \sup_{x\in T}(p, x), \quad s_{K+T}(p) = \sup_{x\in K+T}(p,x)
$$
and note that $s_{K+T}(p) = s_K(p) + s_T(p)$. Therefore $s_K(p) - s_{K+T}(p)$ is concave. Hence the set of $p\in \mathbb R^n$ such that $s_K(p) >  s_{K+T}(p)$ is either empty (in which case $K\setminus (K+T)$ is also empty) or an open convex cone $C$. Consider the subset $K'\subset K$, where the linear functions $p\in C$ attain their maxima on $K$. From the smoothness and strict convexity it follows that this set is homeomorphic to an intersection of $C$ with the unit sphere, and is therefore contractible. 

After this, for every $x\in K'$ it is possible to choose a ray $\rho_x$ from $x$ in such a way that it depends continuously on $x$, no pair of such rays intersect each other, every point in $(\inte K)\setminus(K+T)$ belongs to one such ray, and the intersection of $(\inte K)\setminus(K+T)$ with every $\rho_x$ is an interval depending continuously on $x$. This all allows to conclude that $(\inte K)\setminus(K+T)$ is homotopy equivalent to $K'$, and is therefore contractible. 

The choice of $\rho_x$ in~\cite{kar2001car} was a bit lengthy, but here we want to outline a simpler construction. For any linear function $p\in C$, we take $x(p)$ to be the point where $p$ attains its maximum on $K$, and $y(p)$ to be the point where $p$ attains its maximum on $K+T$. The ray $\rho(p)$ will just start at $x(p)$ and pass through $y(p)$, it is easy to check that its direction is in the interior of the polar cone of $C$. For any two $p', p''\in C$ the two rays $\rho(p'), \rho(p'')$ do not intersect, in order to establish this it is sufficient to consider the projection of the picture to the two-dimensional plane with coordinates $p'$ and $p''$. The fact that every point $z\in \inte K$ has such a ray passing through $z$ is established by considering a homothet $H_z^{-t}(K+T)$ with center $z$ and big negative coefficient $-t$ and decreasing $t$ until $H_z^{-t}(K+T)$ touches $K$ at a point $x$, the point $y$ is then $H_z^{-1/t}(x)$. Eventually, we can retract the set $(\inte K)\setminus(K+T)$ to $K'$ (the set of all possible $x(p)$) along these rays to show its contractibility.
\end{proof}


\begin{proof}[Proof of Theorem~\ref{theorem:colorfulcovering}] Take a system of representatives $x_i\in X_i$. The assumption that $K-t\supset\{x_0, x_1,\ldots, x_n\}$ reads as $t\in K-x_i$ for every $i=0, 1,\ldots, n$, and the assumption $K-t\not\ni 0$ reads as $t\not\in K$. Now put 
$$
F_{x_i} = (K-x_i)\setminus K.
$$
These sets form $n+1$ families $\mathcal F_i$ indexed by $i$. The hypothesis of the theorem now reads: For any system of representatives $F_i\in \mathcal F_i$, their intersection is nonempty. 

Now we show that any intersection of $F_i$'s is either empty or acyclic. Indeed, this set is obtained as follows: First, we intersect a family of translates of $K$ to obtain $K\gdiv T$ (in the notation of the introduction); second, we subtract $K$ to obtain $(K\gdiv T)\setminus K$. But Definition~\ref{definition:genset} means that $K = (K\gdiv T) + T'$ for a convex compactum $T'$, and Lemma~\ref{lemma:acyclic} concludes that the set $(K\gdiv T)\setminus K$ is either empty of acyclic. Thus the colorful topological Helly theorem is applicable, and for some $i$, the sets $\{K-x_i\}_{x_i\in X_i}$ have a common point outside of $K$, which is equivalent to the conclusion of the theorem.
\end{proof}

\section{More Carath\'{e}odory-type statements}
\label{section:more}

It is possible to generalize Theorem~\ref{theorem:colorfulcovering} slightly, by replacing the requirement of not touching the origin with not touching a given convex compactum. Let us say that a translate $K-t$ \emph{separates} a set $S$ from a convex compactum $C$ if $K-t$ contains the set $S$ and is disjoint from $C$.

\begin{theorem}
\label{theorem:colorfulcovering2}
Let $X_0, X_1, \ldots, X_n\subset \mathbb R^n$ be finite sets, $K$ a generating set in $\mathbb R^n$, and $C$ a convex compactum. Suppose every system of representatives $x_0\in X_0$, $x_1\in X_1$, $\ldots$, $x_n\in X_n$ can be separated from $C$ by a translate $K-t$. Then there exists a translate of $K$ that separates $X_i$ from $C$, for some $0\leq i \leq n$.
\end{theorem}

\begin{proof}
The proof proceeds like the proof of Theorem~\ref{theorem:colorfulcovering}, putting
$$
F_{x_i} = (K-x_i)\setminus (K+(-C)).
$$
Again, any intersection of translates $K\gdiv T$ is a Minkowski summand of $K+(-C)$, thus the colorful topological Helly theorem is applicable. 
\end{proof}

The following result improves Theorem~\ref{theorem:colorfulcovering} in the case when the union of the $X_i$ is covered by the interior of a single translate of $K$. In this case we obtain a generalization of the version of the colorful Carath\'eodory theorem (``very colorful theorem'') appearing in \cite{arocha2009very, holmsen2008surrounding}:

\begin{theorem} 
\label{theorem:strong-caratheodory}
Let $X_0$, $X_1$, $\dots$, $X_n \subset \mathbb{R}^n$ be non-empty finite sets and $K$ a generating set in $\mathbb{R}^n$ containing the origin and the sets $X_0, X_1, \dots, X_n$ in its interior. Suppose every system of representatives $x_0\in X_0$, $x_1\in X_1$, $\ldots$, $x_n\in X_n$ can be separated from the origin by a translate $K-t$. Then there exists a translate of $K$ that separates $X_i\cup X_j$ from the origin, for some $0\leq i < j \leq n$.
\end{theorem}

The proof of Theorem \ref{theorem:strong-caratheodory} uses the following result, which is a special case from \cite{holmsen2013}. Let $\mathcal C$ be a simplicial complex. For a simplex $\sigma \in \mathcal C$, the link of $\sigma$ is the simplicial complex defined as 
\[
\lk_{\mathcal C}(\sigma) = \{\tau \in \mathcal C : \tau \cap \sigma = \emptyset, \tau \cup \sigma \in \mathcal C\}.
\]
Here and in the following section we let $\tilde{H}_*(\mathcal C)$ denote the reduced homology with rational coefficients.

\begin{proposition} 
\label{proposition:leray-link}
Let $\mathcal C$ be a simplicial complex on the vertex set $V$ which satisfies the following:
  
  \begin{enumerate}
  \item The vertices are partitioned into non-empty parts $V = V_0 \cup V_1 \cup \cdots \cup V_n$.
  
  \item $\mathcal C$ contains every system of representatives (transversal for short) $v_0\in V_0, v_1\in V_1, \dots, v_n\in V_n$ of the partition. That is, every transversal $\{v_0, v_1, \dots ,v_n\}$ is a simplex in $\mathcal C$. 
  
  \item For all $i\geq n$, we have $\tilde{H}_i(\mathcal C) = 0$.
  
  \item For every non-empty simplex $\sigma \in \mathcal C$ and $i\geq n-1$, we have $\tilde{H}_i(\lk_{\mathcal C}(\sigma)) = 0$.
  \end{enumerate}
Then there exist indices $0\leq i < j \leq n$ such that $V_i\cup V_j$ is a simplex in $\mathcal C$.
\end{proposition}

Let us first show how to deduce Theorem \ref{theorem:strong-caratheodory} from Proposition \ref{proposition:leray-link}, the proof of the latter will be given in the next section.

\begin{proof}[Proof of Theorem \ref{theorem:strong-caratheodory}] 
We define a simplicial complex $\mathcal C$ on the vertex set $V = X_0\sqcup X_1 \sqcup \cdots \sqcup X_n$. By this we mean that each point corresponds to a vertex, but we keep track of multiplicities in the sense that a point which appears in both $X_i$ and $X_j$ appears as distinct vertices in $V$. The simplices of $\mathcal C$ are the subsets of $X_0\sqcup X_1 \sqcup \cdots \sqcup X_n$ which can be separated from the origin by a translate of $K$. 

Note that by the hypothesis we may assume that every point is distinct from the origin and can therefore be separated by a translate of $K$, so every point corresponds to a vertex of $\mathcal C$. We will interchange freely between referring to vertices of $\mathcal C$ and points in $X_0\sqcup X_1 \sqcup \cdots \sqcup X_n$. Note that the hypothesis implies that the simplicial complex $\mathcal C$ satisfies conditions (1) and (2) of Proposition \ref{proposition:leray-link}, so to complete the proof it remains to verify that $\mathcal C$ satisfies conditions (3) and (4).

For a point $x\in V$, let $K_x = (K-x)\setminus K$. As in the proof of Theorem \ref{theorem:colorfulcovering}, a subset $S\subset V$ is separated from the origin by a translate $K-t$ if and only if the set $\bigcap_{x\in S} K_x$ is non-empty. Thus, the simplicial complex $\mathcal C$ is just the nerve of the family $\{K_x\}_{x\in V}$. We now use the nerve theorem (justified in Section~\ref{section:proof1}) to verify conditions (3) and (4).

Condition (3): We know from Lemma~\ref{lemma:acyclic} that for any subset $S\subset V$ the set $\bigcap_{x\in S} K_x$ is empty or acyclic. Thus we may apply the nerve theorem, which implies that $\mathcal C$ has the same homology as $\bigcup_{x\in V} K_v \subset \mathbb R^n\setminus K$. This shows that $\tilde{H}_i(\mathcal C)=0$ for all $i\geq n$.

Condition (4):
For each $x\in V$ we define a subset $L_x$ contained in the boundary of $K$ as
\[
L_x = \left\{\pi(v) : v \in K_x \right\},
\] 
where $\pi$ denotes the central projection from the origin to the boundary of $K$. 

The crucial observation is that for any $S\subset V$ we have
\[
\bigcap_{x\in S} L_{x} \neq \emptyset \iff \bigcap_{x\in S} K_{x} \neq \emptyset.
\]
This follows because, for every $x\in V$, the translate $K-x$ contains the origin in its interior (this is assumed in the statement of the theorem) and every point in $\bigcup_{x\in V}K_x$ can be seen from the origin, that is, it is strictly starshaped. Hence the preimage of any $p\in L_x$ in $K_x$ under the map $\pi$ is a semi-interval starting from some $q\in\partial (K-x)$ and ending in $p$ (not containing $p$). Therefore, a nonempty intersection of several of $L_x$'s implies a nonempty intersection of the corresponding set of $K_x$'s; the opposite is trivially true.

Let $\sigma$ be a non-empty simplex of $\mathcal C$. The vertices of $\lk_{\mathcal C}(\sigma)$ correspond to the points $y\not\in \sigma$ such that $L_y \cap \left(\bigcap_{x\in \sigma}L_x\right) \neq\emptyset$. As before, we know that for every subset of vertices $S$ in $\lk_{\mathcal C}(\sigma)$, the set 
\[
\left(\bigcap_{y\in S}L_y\right) \cap \left(\bigcap_{x\in \sigma}L_x\right)
\] 
is empty or acyclic. So applying the nerve theorem to the family $\left\{L_y\cap \left(\bigcap_{x\in \sigma}L_x\right)\right\}$, where $y$ ranges over the vertices of $\lk_{\mathcal C}(\sigma)$, we see that $\lk_{\mathcal C}(\sigma)$ is homotopy equivalent to $\left(\bigcup L_y\right)\cap \left(\bigcap_{x\in \sigma}L_x\right) \subset \bigcap_{x\in \sigma}L_x$. Since $\bigcap_{x\in \sigma}L_x$ is an acyclic subset of the boundary of $K$, and the boundary of $K$ is homemorphic to $\mathbb{S}^{n-1}$, it follows that $\tilde{H}_i(\lk_{\mathcal C}(\sigma)) = 0$ for all $i  \geq (n-1)$. 

Now all the conditions of Proposition \ref{proposition:leray-link} are satisfied, so Theorem \ref{theorem:strong-caratheodory} is proved.
\end{proof}

\section{Proof of Proposition~\ref{proposition:leray-link}}
\label{section:leray-link}

For completeness, we now give a proof of Proposition~\ref{proposition:leray-link}. The main tool used in this proof is a Sperner-type lemma due to Meshulam \cite{meshulam2001}, which also appears in \cite{km2005}. 

Let $\mathcal C$ be a simplicial complex with vertex set $X$ which is partitioned into non-empty parts $X = X_1 \sqcup X_2 \cup \cdots \sqcup X_N$. A \emph{colorful simplex} in $\mathcal C$ is an $(N-1)$-simplex which is a transversal of the partition. For a non-empty subset $I\subset \{1,2,\dots, N\}$, let $\mathcal C[I]$ denote the subcomplex of $\mathcal C$ induced by the vertices $\bigcup_{i\in I}X_i$. Meshulam's lemma gives a sufficient condition for the existence of a colorful simplex in $\mathcal C$.

\begin{lemma}
\label{lemma:meshulam-colorful}
Let $\mathcal C$ be a simplicial complex with vertex set $X$ which is partitioned into non-empty parts $X = X_1$ $\sqcup$ $X_2$ $\sqcup \cdots \sqcup$ $X_N$. Suppose for every non-empty subset $I\subset \{1,2,\dots, N\}$, we have \[\tilde{H}_i(C[I]) = 0 , \text{ for all }  i \leq |I|-2.\] Then $\mathcal C$ contains a colorful simplex.
\end{lemma}

\begin{proof}[Proof of Proposition \ref{proposition:leray-link}]
Let $V = \{v_1, v_2, \dots, v_N\}$ denote the vertices of $\mathcal C$ and let $W=\{w_1, w_2, \dots, w_N\}$ be a disjoint copy of $V$. The partition $V = V_0\sqcup V_1 \sqcup \cdots \sqcup V_n$ induces a partition $W = W_0\sqcup W_1\sqcup \cdots \sqcup W_n$. Let $\mathcal D$ denote the simplicial complex on $W$ whose simplices consist of all partial transversals of this partition. Condition (2) therefore reads that for every simplex $\sigma_W = \{w_{i_1}, w_{i_2}, \dots, w_{i_k}\}$ in $\mathcal D$, the corresponding simplex $\sigma_V = \{v_{i_1}, v_{i_2}, \dots, v_{i_k}\}$ is in $\mathcal C$.

For a subset $S\subset V$ let $r(S)$ denote the cardinality of the inclusion-minimal subset $I\subset \{0,1,\dots, n\}$ such that $S\subset \bigcup_{i\in I}V_i$, that is, the minimal number of parts of the partition which contain $S$. The same notation will be used for subsets of $W$.

We prove the contrapositive. Suppose for every $0\leq i< j \leq n$ the set $V_i\cup V_j$ is not a simplex in $\mathcal C$, or equivalently, for every simplex $S$ in $\mathcal C$ we have $r(V\setminus S) \geq n$. We then show that there is a transversal of the partition which is not a simplex in $\mathcal C$, which violates condition (2).

Let $\mathcal C^\star$ denote the Alexander dual of $\mathcal C$, which is defined as 
\[
\mathcal C^\star = \{S\subset V : V\setminus S \not\in \mathcal C\}.
\] 
Define the simplicial complex $\mathcal B$ on vertices $X = V\sqcup W$ as the join of $\mathcal C^\star$ and $\mathcal D$, that is, $\mathcal B = \mathcal C^\star * \mathcal D$. The vertex set of $\mathcal B$ has a natural partition into $N$ non-empty parts $X = X_1 \sqcup X_2 \sqcup \cdots \sqcup X_N$ where  $X_i = \{v_i,w_i\}$, that is, each part consists of a vertex from $V$ and its copy in $W$.

Our goal is to use Lemma \ref{lemma:meshulam-colorful} to show that $\mathcal B$ has a colorful simplex. This will complete the proof because a colorful simplex in $\mathcal B$ corresponds to disjoint subsets $I$ and $J$ such that  $\{v_i\}_{i\in I}$ is a simplex in $\mathcal C^\star$, $\{w_j\}_{j\in J}$ is a simplex in $\mathcal D$, and  
\[
I\cup J = \{1,2,\dots, N\}.
\] 
However, $\{v_i\}_{i\in I}$ is a simplex in $\mathcal C^\star$ if and only if $\{v_j\}_{j\in J}$ is not a simplex in $\mathcal C$, which gives us the desired contradiction to condition (2) since $\{w_j\}_{j\in J}$ is a simplex in $\mathcal D$ and therefore a partial transversal of the partition $W= W_0\sqcup W_1 \sqcup \cdots \sqcup W_n$ (and therefore also of the corresponding partition of $V$). So it remains to verify that $\mathcal B$ satisfies the conditions of Lemma \ref{lemma:meshulam-colorful}.

For  $\emptyset \neq I \subset \{1,2,\dots, N\}$ let $S = \{v_i\}_{i\in I}$ and $T = \{w_i\}_{i\in I}$. We want to show that  $\tilde{H}_i(\mathcal B[I]) = 0$ for all $0\leq i \leq |I|-2$. If $S$ is a simplex in $\mathcal C^\star$, then $\mathcal C^\star[S]$ is acyclic which implies $\tilde{H}_i(\mathcal C^\star[S] * \mathcal D[T]) = 0$ for all $i$. So we assume that $S$ is not a simplex in $\mathcal C^\star$ and consequently $V\setminus S$ is a simplex in $\mathcal C$.

By the K\"unneth formula for the join, we have

\begin{equation} \label{eq:kunneth}
  \tilde{H}_i(\mathcal B[I])  \cong  \tilde{H}_i(\mathcal C^\star[S] * \mathcal D[T]) \cong  \bigoplus_{k+l = i-1} \tilde{H}_k(\mathcal C^\star[S])  \otimes \tilde{H}_l(\mathcal D[T]) .
\end{equation}

Clearly we may assume that $\mathcal C$ is not a simplex, so Alexander duality implies

\begin{equation}\label{eq:alexander}
\tilde{H}_k(\mathcal C^\star[S]) \cong \tilde{H}_{|S|-k-3}(\lk_{\mathcal C}(V\setminus S)).
\end{equation}

Moreover, if $r = r(S) = r(T)$, then $\mathcal D[T]$ is a join of $r$ disjoint sets of isolated vertices, and consequently $\mathcal D[T]$ is acyclic or $(r-2)$-connected, and therefore \[\tilde{H}_l(\mathcal D[T]) =0 \text{ for all }  l\leq r-2\]

We now consider two cases:
\begin{enumerate}
\item 
If $I = \{1,2,\dots, N\}$, then $S = V$, so $\mathcal C^\star[S] = \mathcal C^\star$ and $\lk_{\mathcal C}(\emptyset) = \mathcal C$. By condition (3) we have
\[
\tilde{H}_i(\mathcal C) = 0 \text{ for all } i\geq n,
\]
so by Alexander duality  \eqref{eq:alexander}
\[
\tilde{H}_{k}(\mathcal C^\star) = 0 \text{ for all } k\leq N-n-3.
\]
In this case we also have $T = W$, so $r(T) = n+1$ and $\mathcal D[T] = L$, which implies that 
\[
\tilde{H}_l(\mathcal D) = 0 \text{ for all } l\leq n-1.
\] 
Thus the K\"unneth formula \eqref{eq:kunneth} implies that $\tilde{H}_i(\mathcal B) =0$ for all $i\leq N-2$.

\item 
If $I$ is a proper subset of $\{1,2,\dots, N\}$, then $\sigma = V\setminus S$ is a non-empty simplex of $\mathcal C$. By condition (4) we have
\[
\tilde{H}_i(\lk_{\mathcal C}(\sigma)) = 0 \text{ for all } i\geq n-1,
\]
so by Alexander duality \eqref{eq:alexander}
\[
\tilde{H}_{k}(\mathcal C^\star[S]) = 0 \text{ for all } k\leq |S|-n-2.
\]
Since $\sigma$ is a simplex in $\mathcal C$, it follows from our assumption that $r(S) = r(T)\geq n$, which implies that 
\[
\tilde{H}_l(\mathcal D[T]) = 0 \text{ for all } l\leq n-2.
\]
Thus the K\"unneth formula \eqref{eq:kunneth} implies that $\tilde{H}_i(\mathcal B[I]) =0$ for all $i\leq |I|-2$.
\end{enumerate}

This shows that the conditions of Lemma \ref{lemma:meshulam-colorful} are met, and the proof is complete.
\end{proof}

\section{No Carath\'eodory-type theorem for arbitrary convex sets}
\label{section:no-cara}

The above proofs use the nontrivial colorful topological Helly theorem~\cite{km2005} and its relatives, like Proposition~\ref{proposition:leray-link}. It is interesting whether it is possible to make this argument more elementary, or give a different proof. For example, is it possible to prove such results with optimization ideas like in~\cite{ba1982}? 

Moreover, we may ask whether the ``generating set'' property is really needed in the argument; the usage of such a property here was dictated by the topological Helly-type theorems which require us to have acyclic intersections. 

In this section 
 we show that for dimension greater or equal to 3, the Carath\'eodory theorem does not holds for $K$-strong convexity when the set $K$ is an arbitrary convex body. In particular, for every positive integer $n$, we give an example of a set $X$ of $2n$ points and a convex set $K$ in $\mathbb{R}^3$ such that every proper subset of $X$ can be separated from the origin by a translate of $K$, but where no translate of $K$ separates the entire set $X$ from the origin. The convex set $K$ will be an epigraph of a function $f(x,y)$ and will therefore not be compact, but it is easy to produce a compact example by intersecting it with the halfspace $\{z \; : \; z \le C\}$ for a sufficiently large $C$.

Our point set $X$ will consist of the points for $k=1,\ldots, n$
$$
\xi_{\pm k} = (\pm k, 0, k^2).
$$
The point that we are going to test for containment in $\conv_K X$ and $\conv_K Y$ for proper subsets $Y\subset X$ is the origin $\xi_0 = (0, 0, 0)$.

Note that the set $X$ lies in the plane $\{y=0\}$ and on the parabola $z = h(x) := x^2$. We are going to consider $y$ as a parameter and describe how we choose the functions $f(x,y)$ of $x$ for given $y$. Let us choose them so that $f(x,y)$ is a $C^2$ smooth strictly convex (even with $f''_{xx} (x, y) \ge 1$) function of $x$, always satisfying the assumption 
$$
f(0,y) = 0,\quad f(\pm k, y) \le h(\pm k) = k^2, 
$$
for $k=1,\ldots, n$. When the equality in the latter inequality holds, we will call this ``\emph{$f(x,y)$ touches $\xi_{\pm k}$}''.

Now we require that $f(x,y)$ depends $C^2$ smoothly on $y$, is even in $y$, and satisfies the following assumption: $f(x,y)$ touches $\xi_{\pm k}$ precisely when $|y|\in [k-1, k]$ for $k<n$, and precisely for $|y|\in [n-1, +\infty)$ when $k=n$; moreover, let $f(x,y)$ be independent of $y$ for $|y|\ge n$. It is clear that such a family of functions exists and may be chosen twice continuously differentiable in $y$. The total function $f(x,y)$ is still not necessarily convex, but it can be adjusted to become convex after considering 
$$
f(x,y) + Ay^2
$$
with sufficiently large $A$ instead. This adjustment changes the restriction to the plane $y=\const$ only by adding a constant, which is not relevant to finding $\conv_K X$, since $y$ is constant on $X$. So we prefer to consider here the non-modified and intentionally non-convex $f$ for clarity.

From the ``touching'' assumption it is easy to conclude that $\xi_0$ is in $\conv_K X$. Indeed, if the epigraph is translated with some value $-y$, then we actually consider the convex hull of the planar set $X$ with respect to the set $K_y$, the epigraph of $f(x,y)$ with fixed $y$. To put it clearly,
$$
\conv_K X\cap \{y=0\} = \bigcap_y \conv_{K_y} X.
$$
Again, from the ``touching'' assumption it is relatively clear that $\xi_0$ is in $\conv_{K_y} X$ for every $y$. But, the touching assumption also guarantees that for $|y|\in (k-1, k)$ the only thing preventing $\xi_0$ from getting outside $\conv_{K_y} X$ is the pair of points $\xi_{\pm k}$. In this range, if $\xi_{-k}$ and $\xi_k$ are in the translated $K_y$ then $\xi_0$ gets into this translate of $K_y$ as well. But, if we drop one of the points $\xi_{\pm k}$ ($\xi_k$ without loss of generality) then it immediately becomes possibly to cover $X\setminus\{\xi_k\}$ with a translate of $K_y$ leaving $\xi_0$ outside. This means that $\conv_K X$ fails to contain $\xi_0$ when any one of the points of $X$ is dropped.

\begin{remark}
The above example shows that the Carath\'eodory number is not bounded from above by a constant independent of the body $K\subset\mathbb R^3$. It can be modified by considering an infinite sequence 
$$
\xi_{\pm k} = (\pm (2 - 1/k), 0, (2 - 1/k)^2),
$$ 
 showing that for a particular body $K$ the Carath\'eodory number can be infinite. 
\end{remark}

To conclude this discussion, we ask:

\begin{question}
Is the existence of a finite Carath\'eodory number for $K$-strong convexity equivalent to the \emph{generating} property of $K$? 
\end{question}

\begin{question}
Is the property that the Carath\'eodory number for $K$-strong convexity equals $n+1$ equivalent to the \emph{generating} property of $K\subset\mathbb R^n$?
\end{question}

\section{Topological criterion for Minkowski summand}
\label{section:topo}

The crucial topological tool used in the proofs of Theorems \ref{theorem:colorfulcovering}, \ref{theorem:colorfulcovering2}, and \ref{theorem:strong-caratheodory} was Lemma \ref{lemma:acyclic}. In this section we investigate whether its converse also holds. Note that Lemma~\ref{lemma:acyclic} can be reformulated as follows: Assume a convex body $A\subset\mathbb R^n$ is a Minkowski summand of another convex body $B\subset\mathbb R^n$; then for every vector $t\in\mathbb R^n$ the set $(A+t)\setminus B$ is either empty or acyclic. In this section we are going to obtain a certain inverse theorem.





The crucial fact that we need is the following:

\begin{lemma}[Vietoris--Begle]
\label{lemma:leray-map}
Let $f : X\to Y$ be a proper continuous map between metric spaces such that for every $y\in Y$, the fiber $f^{-1}(y)$ is acyclic. Then $f$ induces an isomorphism of the \v{C}ech cohomology of $X$ and $Y$.
\end{lemma}

This is a well-known fact, see~\cite{begle1950,begle1955}; but we sketch the proof for reader's convenience. We remind the reader once again, that in our situation \v{C}ech cohomology is the same as sheaf cohomology.

\begin{proof}
Consider the Leray spectral sequence for the direct image of a sheaf with~\cite[Theorem 4.17.1]{godement1958} 
$$
E_2^{p,q} = H^p(X; \mathcal H^q(f)), 
$$
where the sheaf $\mathcal H^q(f)$ is generated by the presheaf $U\mapsto H^q(f^{-1}(U))$. This spectral sequence converges to $H^*(Y)$. But our assumption on the fibers together with the cohomology continuity property shows that the sheafs $\mathcal H^q(f)$ are trivial for $q>0$, and $\mathcal H^0(f)$ is the constant sheaf $\mathbb Z$. Hence in this spectral sequence $E_2 = H^*(X)$ and its differentials vanish, and therefore $H^*(Y) = E_\infty = H^*(X)$.
\end{proof}

Now we state the result:

\begin{theorem}
\label{theorem:acyclic}
Let $A$ be a convex body in $\mathbb R^n$ and $B$ be a convex open bounded set in $\mathbb R^n$. Assume that, for any vector $t\in\mathbb R^n$, the set $(A+t)\setminus B$ is either empty or acyclic. Then $A$ is a Minkowski summand of $B$.
\end{theorem}


\begin{proof}
Put $C = B\gdiv A$, we need to demonstrate that $A+C = B$. Evidently, $C$ is a convex open set.

First, we need to show that $C$ is not empty, that is, a translate of $A$ is contained in $B$. Put
$$
Z = \{(x, t) \in \mathbb R^n\times \mathbb R^n: x\in (A+t)\setminus B\}.
$$
Observe that the projection $\pi_1(Z)$ to the first factor is $\mathbb R^n\setminus B$, while its projection $\pi_2(Z)$ to the second factor is $\mathbb R^n\setminus C$. 

The fiber of the first projection is $\pi_1^{-1}(x) = \{x\}\times (x - A)$, it is compact, convex, and therefore acyclic; the fiber of the second projection $\pi_2^{-1}(t) = ((A+t)\setminus B)\times \{t\}$ is compact and acyclic by the hypothesis. Also, it is easy to check by definition that both projections of $Z$ are proper maps. Thus Lemma~\ref{lemma:leray-map} applies to both projections and implies that both projections induce isomorphisms in cohomology. Since 
$H^{n-1}(\mathbb R^n\setminus B) = \mathbb Z$, then we must have $H^{n-1}(\mathbb R^n\setminus C) = \mathbb Z$. Therefore the set $C$ must be nonempty.

Now we continue to work with $Z$ and consider its subset 
$$
Z(\varepsilon) = \{(x,t)\in Z : \dist (t, C) < \varepsilon)\}.
$$
Obviously $\pi_2(Z(\varepsilon))$ is $C_\varepsilon\setminus C$; therefore its cohomology is the same as the cohomology of an $(n-1)$-dimensional sphere, that is $H^{n-1}(C_\varepsilon\setminus C) = \mathbb Z$. Lemma~\ref{lemma:leray-map} concludes that $H^{n-1}(Z(\varepsilon)) = \mathbb Z$ also.

Consider the commutative diagram:
$$
\begin{CD}
Z(\varepsilon) @>>> Z\\
@VV{\pi_2}V @VV{\pi_2}V\\
C_\varepsilon\setminus C @>>> \mathbb R^n\setminus C,
\end{CD}
$$
where the horizontal arrows are inclusions. Since the vertical arrows and the lower horizontal arrow induce isomorphisms in cohomology, the same is true for the upper horizontal arrow.

Now take a point $o$ in the interior of $B$ and consider the central projection to the unit sphere centered at $o$
$$
\phi : \mathbb R^n\setminus B \to \mathbb S^{n-1}.
$$
Note that $\phi$ obviously induces isomorphism in the cohomology. Hence the composition
$$
\begin{CD}
Z(\varepsilon) @>>> Z @>{\pi_2}>> \mathbb R^n\setminus B @>{\phi}>> \mathbb S^{n-1},
\end{CD}
$$
which we denote by $\psi$, induces an isomorphism $\psi^* : H^{n-1}(\mathbb S^{n-1})\to H^{n-1}(Z(\varepsilon))$. This implies $\psi$ must be surjective, because an inclusion $\mathbb S^{n-1}\setminus \nu \to \mathbb S^{n-1}$ (for any $\nu\in \mathbb S^{n-1}$) induces a zero map in $(n-1)$-dimensional cohomology.

Now let us decode the geometric meaning of what we have proved: For any unit vector $\nu$ the ray 
$$
\rho_\nu = \{o+\lambda \nu : \lambda \ge 0\}
$$ 
intersects the set $(A+t)\setminus B$ for some $t\in C_\varepsilon$. Now we let $\varepsilon$ tend to $0$ and use the compactness to conclude that any $\rho_\nu\cap \partial B$ is contained in $A+t$ for some $t$ in $\cl C$, the closure of $C$. Since any point in $\partial B$ is $\rho_\nu\cap \partial B$ for a suitable $\nu$, we conclude that the convex set $A+\cl C$ contains the boundary of $B$ and therefore contains $\cl B$. From this it is easy to see that $A+C$ must coincide with the whole set $B$.
\end{proof}

To conclude let us put forth two conjectures. Lemma~\ref{lemma:acyclic} asserts that $A\setminus (A+B)$ is either empty or acyclic. Can one make an analogous assertion concerning the set $(A+B)\setminus A$?

\begin{conjecture}
Let $A$ and $B$ be convex bodies in $\mathbb R^n$. Then the set  $(A+B)\setminus A$ is either empty, or acyclic, or has homology of a sphere of dimension $k\in [0, n-1]$.
\end{conjecture}

In Theorem \ref{theorem:acyclic} we needed $B$ to be open because it was convenient to have the sets $(A+t)\setminus B$ compact. Can one obtain the same conclusion assuming that $B$ is a convex body? 

\begin{conjecture}
Let $A$ and $B$ be convex bodies in $\mathbb{R}^n$. Assume that, for any vector $t\in \mathbb{R}^n$, the set $(A+t) \setminus B$ is either empty or acyclic. Then $A$ is a Minkowski summand of $B$.
\end{conjecture}



\end{document}